\documentclass[11pt]{amsart}

\usepackage{amsmath,amssymb,latexsym,soul,cite,mathrsfs}

\usepackage{color,enumitem,graphicx}
\usepackage[colorlinks=true,urlcolor=blue,
citecolor=red,linkcolor=blue,linktocpage,pdfpagelabels,
bookmarksnumbered,bookmarksopen]{hyperref}
\usepackage[english]{babel}

\usepackage[left=2.9cm,right=2.9cm,top=2.8cm,bottom=2.8cm]{geometry}
\usepackage[hyperpageref]{backref}

\usepackage[colorinlistoftodos]{todonotes}
\makeatletter
\providecommand\@dotsep{5}
\def\listtodoname{List of Todos}
\def\listoftodos{\@starttoc{tdo}\listtodoname}
\makeatother

\numberwithin{equation}{section}
\newtheorem{theorem}{Theorem}[section]

\newtheorem{lemma}[theorem]{Lemma}
\newtheorem{corollary}[theorem]{Corollary}

\newtheorem{claim}[theorem]{Claim}

\newcommand\R{\mathbb R}

\begin{document}

\title[A Lions type result for a large class ...]
{A Lions type result for a large class of Orlicz-Sobolev space and applications}
\author{Claudianor O. Alves$^*$ and Marcos L. M. Carvalho}
\address[Claudianor O. Alves ]
{\newline\indent Unidade Acad\^emica de Matem\'atica
	\newline\indent 
	Universidade Federal de Campina Grande 
\newline\indent
e-mail:coalves@mat.ufcg.edu.br
\newline\indent	
58429-970, Campina Grande - PB, Brazil} 

\address[ Marcos L.M. Carvalho]
{\newline\indent Instituto de Matem\'atica e Estat\'istica 
	\newline\indent 
	Universidade Federal de Goias 
\newline\indent
e-mail: marcos\_leandro\_carvalho@ufg.br
\newline\indent 74001-970, Goi\^ania, GO, Brazil }

\pretolerance10000

%\begin{document}

\begin{abstract}
In this paper we prove a Lions type result for a large class of Orlicz-Sobolev space that can be nonreflexive and use this result to show the existence of solution for a large class of quasilinear problem on a nonreflexive  Orlicz-Sobolev space.
\end{abstract}

\thanks{ C. O. Alves is the corresponding author and he was partially
	supported by  CNPq/Brazil 304804/2017-7 }

\maketitle	

\section{Introduction}

In recent years, a special attention has been given for quasilinear problems of the type 
$$
\left\{
\begin{array}{l}
	- \mbox{div}(\phi(|\nabla u|)\nabla u) +  V(x)\phi(|u|)u=f(u), \quad  \mbox{in} \ \mathbb{R}^N, \\
	\mbox{}\\
	u \in W^{1,\Phi}(\mathbb{R}^N)  \,\,\, \mbox{with} \,\,\, N \geq 1,
\end{array}
\right.
\eqno{(P)}
$$
where $V,f$ are continuous functions satisfying some technical  conditions and \linebreak $\phi:(0,+\infty) \to (0,+\infty)$ is a $C^1$-function.  This type of problem appears in a lot of physical applications, such as Nonlinear Elasticity, Plasticity, Generalized Newtonian Fluid, Non-Newtonian Fluid and Plasma Physics. For more details involving this subject see \cite{Db}, \cite{FN2} and their references.

We cite the papers of Bonanno, Bisci and Radulescu \cite{BBR, BBR2}, Cerny \cite{Cerny}, Cl\'ement, Garcia-Huidobro and Man\'asevich \cite{VGMS}, Donaldson \cite{donaldson}, Fuchs and Li \cite{Fuchs1},  Fuchs and Osmolovski \cite{Fuchs2}, Fukagai, Ito and Narukawa \cite{FN}, Gossez \cite{gossez2}, Le and Schmitt \cite{LK}, Mihailescu and Radulescu \cite{MR1, MR2}, Mihailescu and Repovs \cite{MD}, Mihailescu, Radulescu and Repovs \cite{MRR}, Orlicz \cite{O} and their references, where quasilinear problems like $(P)$ have been considered in bounded and unbounded domains of $\mathbb{R}^{N}$. 

For the particular where $\phi(t)=|t|^{p-2}$ with $1<p<N$, problem $(P)$ becomes a $p$-Laplacian problem of the form
$$
\left\{
\begin{array}{l}
	-\Delta_p u +  V(x)|u|^{p-2}u=f(u), \quad  \mbox{in} \ \mathbb{R}^N, \\
	\mbox{}\\
	u \in W^{1,p}(\mathbb{R}^N).
\end{array}
\right.
\eqno{(P_p)}
$$
When $V$ is $\mathbb{Z}^N$ periodic, it is easy to check that the energy functional $J:W^{1,p}(\mathbb{R}^N) \to \mathbb{R}$ associated with $(P_p)$ given by 
$$
J(u)=\frac{1}{p}\int_{\mathbb{R}^N}(|\nabla u|^p+V(x)|u|^p)\, dx -\int_{\mathbb{R}^N}F(u)\,dx
$$
is invariant by $\mathbb{Z}^N$-translation. Have this in mind, in general the main tool used to get a nontrivial critical point to  $J$ is the following result due to Lions \cite{Lions}
\begin{theorem} \label{lions} Let $r>0$ and $1\leq p\leq q <p^*$. If $(u_n) \subset W^{1,p}(\mathbb{R}^N)$ is a bounded sequence and  
$$
\lim_{n \to +\infty}\sup_{y \in \mathbb{R}^N}\int_{B_r(y)}|u_n|^{q}\,dx=0,
$$
then $u_n \to 0$ in $L^{t}(\mathbb{R}^N)$ for all $t \in (p,p^*)$.	
\end{theorem}

The lemma above is used of the following way: If $(u_n)$ is a $(PS)_c$ sequence  for $J$  with $c>0$, assuming good conditions on $f$, it is possible to prove that Theorem \ref{lions} does not hold, and so, there must be $(y_n) \subset \mathbb{Z}^N$ and $\beta>0$ such that 
$$
\int_{B_r(y_n)}|u_n|^q\,dx\geq \beta, \quad \forall n \in \mathbb{N}. 
$$ 
After that, we consider the sequence $v_n(x)=u_n(x+y_n)$ that has a nontrivial weak limit, which is in general a nontrivial critical point to $J$. For more details involving this type of argument we cite Alves, do \'O and Miyagaki \cite{ADOM} and references therein.  

If the function $\phi$ is a general function such that $\Phi(t)=\int_{0}^{|t|}\phi(s)s\,dx$ is a N-function, problem $(P)$ becomes a quasilinear  problem  of the form
$$
\left\{
\begin{array}{l}
	-\Delta_\Phi u +  V(x)\phi(|u|)u=f(u), \quad  \mbox{in} \ \mathbb{R}^N, \\
	\mbox{}\\
	u \in W^{1,\Phi}(\mathbb{R}^N),
\end{array}
\right.
\eqno{(P)}
$$
where $\Delta_\Phi u=div(\phi(|\nabla u|)\nabla u)$ is the $\Phi$-Laplacian operator. If $V$ is $\mathbb{Z}^N$ periodic, it is also simple to prove that the energy functional $I:W^{1,\Phi}(\mathbb{R}^N) \to \mathbb{R}$ associated with $(P)$, given by 
$$
I(u)=\int_{\mathbb{R}^N}(\Phi(|\nabla u|)+V(x)\Phi(|u|))\, dx -\int_{\mathbb{R}^N}F(u)\,dx
$$
is invariant by $\mathbb{Z}^N$-translation. In order to repeat the same approach explored for the $p$-Laplacian problem, it was necessary to establish a version of Theorem \ref{lions} for the Orcicz-Sobolev spaces $W^{1,\Phi}(\mathbb{R}^N)$, which was proved by Alves, Figueiredo and Santos \cite{AGJ}, and it has the following statement 
\begin{theorem}\label{dolions} ( {\bf A Lions type result for Orlicz-Sobolev spaces} ) \linebreak Assume that $\phi$  satisfies the following conditions: 
	
	\noindent $i)$ The function $\phi(t)t$ is increasing in $(0,+\infty)$, that is,
	$$
	(\phi(t)t)'>0 \,\,\, \forall t>0.
	\eqno{(i)}
	$$
	
	\noindent $ii)$ There exist $l, m \in (1,N)$ such that
	$$
	l\leq \displaystyle\frac{\phi(|t|)t^{2}}{\Phi(t)}\leq m \,\,\, \forall t \not= 0,
	\eqno{(ii)}
	$$
	where $l\leq m < l^{*}$, $l^{*}=\displaystyle\frac{lN}{N-l}$ and $\Phi(t)=\int_{0}^{|t|}\phi(s)s\,dx$. If $(u_{n})\subset W^{1,\Phi}(\mathbb{R}^N)$ is a bounded sequence  such that there exists $R>0$ satisfying
	$$
	\displaystyle \lim_{n\rightarrow +\infty}\displaystyle \sup_{y\in{\mathbb{R}^N}}\int_{B_{R}(y)}\Phi(|u_{n}|)=0, 
	$$
	then for any N-function $B$ verifying $\Delta_{2}$-condition with
	$$
	\displaystyle\lim_{t\rightarrow 0}\displaystyle\frac{B(t)}{\Phi(t)}=0 \eqno{(B_1)}
	$$
	and
	$$
	\displaystyle\lim_{|t|\rightarrow +\infty}\displaystyle\frac{B(t)}{\Phi_{*}(t)}=0,  \eqno{(B_2)}
	$$
	we have
	$$
	u_{n}\rightarrow 0 \ \ \mbox{in} \ \ L^{B}(\mathbb{R}^{N}).
	$$
\end{theorem}	
In the lemma above, $\Phi_{*}$ denotes the Sobolev conjugate function of $\Phi$ defined by
$$
\Phi^{-1}_{*}(t)=\displaystyle\int^{t}_{0}\displaystyle\frac{\Phi^{-1}(s)}{s^{(N+1)/N}} ds \,\,\, \mbox{for} \,\,\, t >0,
$$
when
$$
\displaystyle\int^{+\infty}_{1}\displaystyle\frac{\Phi^{-1}(s)}{s^{(N+1)/N}} ds=+\infty.
$$

In \cite{AGJ}, the authors used Theorem \ref{dolions} to establish the existence of nontrivial solution for $(P)$ supposing that $V$ is $\mathbb{Z}^N$-periodic and $f$ satisfying some technical conditions.  The approach explored follows the same lines as in the $p$-Laplacian case, that is, if $(u_n)$ is a $(PS)_c$ sequence  for $I$  with $c>0$, with good conditions on $f$, it is possible  to find $(y_n) \subset \mathbb{Z}^N$ and $\beta>0$ such that 
$$
\int_{B_r(y_n)}\Phi(|u_n|)\,dx\geq \beta, \quad \forall n \in \mathbb{N}. 
$$ 
After that, the authors consider the sequence $v_n(x)=u_n(x+y_n)$ that has a nontrivial weak limit, which is a nontrivial critical point for $I$.  

Here, it is very important to say that items $(i)-(ii)$ in Theorem \ref{dolions} ensure that $\Phi$ and $\tilde{\Phi}$ satisfy the $\Delta_2$-condition, and so, the space $W^{1,\Phi}(\R^N)$ is reflexive. 

Motivated by the papers cited above, we are led to try to answer the following question: How can we find a solution to problem $(P)$ when $W^{1,\Phi}(\R^N)$ is nonreflexive?  In the present paper we intend to answer this question. The first difficulty is associated with the fact that we cannot use Theorem \ref{dolions}, because it works well only in reflexive spaces. Here, we prove the following Lions type result that can be used when $W^{1,\Phi}(\R^N)$ is nonreflexive. 
\begin{theorem} \label{Lions1}
	(A Lions type result) Let $\Phi$ and $B$ be N-functions such that there exists $\Phi_*$ and 
	$$
	\lim_{t\to 0}\frac{B(|t|)}{\Phi(|t|)}=\lim_{t\to\infty}\frac{B(|t|)}{\Phi_*(|t|)}=0.
	$$
	If $(w_n) \subset W^{1,\Phi}(\mathbb{R}^N)$ is a sequence such that $(\int_{\mathbb{R}^N}\Phi(|w_n|)dx)$ and $(\int_{\mathbb{R}^N}\Phi_*(|w_n|)dx)$ are bounded, and for each $\epsilon>0$ we have
	$$
	mes([|w_n|>\epsilon]) \to 0, \quad \mbox{as} \quad n \to +\infty, \leqno{(*)}
	$$
	then
	$$
	\int_{\mathbb{R}^N}B(|w_n|)\,dx \to 0 \quad \mbox{as} \quad n \to +\infty.
	$$
\end{theorem}

Note that we can apply Theorem  \ref{Lions1} in Orlicz-Sobolev spaces $W^{1,\Phi}(\R^N)$ that can be nonreflexive, because in that result $\Phi$ or $\tilde{\Phi}$ do not need to satisfy the $\Delta_2$-condition, which is a crucial property to guarantee the reflexivity of $W^{1,\Phi}(\mathbb{R}^N)$. Hence, Theorem \ref{Lions1} improves Theorem \ref{dolions}.

Our next result is a Lieb type result for Orlicz-Sobolev space $W^{1,\Phi}(\R^N)$ that works together with Theorem \ref{Lions1} to get a sequence whose the weak limit is nontrivial. 

\begin{theorem}(A Lieb type result) \label{lieb} Let $\Phi \in C^{1}(\mathbb{R},[0,+\infty))$ be a $N$-function and $(u_n) \subset W^{1,\Phi}(\mathbb{R}^N)$ such that $\displaystyle \int_{\mathbb{R}^N}\Phi(|\nabla u_n|)\,dx \leq M$. If there are $\epsilon,\delta>0$ such that
	$$
	mes([|u_n|>\epsilon])\geq \delta, \quad \forall n \in \mathbb{N}, 
	$$	
	then there is $(z_n) \subset \mathbb{Z}^N$ such that $v_n(x)=u_n(x+z_n)$ has a subsequence whose its limit in $L_{loc}^{\Phi}(\R^N)$ is non trivial.  
\end{theorem}

In order to illustrate how we can use Theorems \ref{Lions1} and \ref{lieb}, we will study the existence of nontrivial solution to problem  $(P)$ assuming that $\phi:(0, +\infty) \to (0, +\infty) $ is a continuous function verifying the following conditions:
$$
t \mapsto t\phi(t) \quad \mbox{is increasing for} \quad t>0. \leqno{(\phi_1)}
$$
$$
\lim_{t \to 0}t\phi(t)=0 \quad \mbox{and} \quad \lim_{t \to +\infty}t\phi(t)=+\infty. \leqno{(\phi_2)}
$$
There are $l, m \in [1,N), \, l \leq m \leq l^*$, and 
$$
l \leq \frac {t^2\phi(t)}{\Phi(t)}\leq m,~t>0.  \leqno{(\phi_3)}
$$

Next, we show some examples of functions $\Phi$ that can be considered in the present paper. If $l>1$, we can consider    
$$
\begin{array}{l}
	i) \,\, \Phi(t)=|t|^{p} \,\,\, \mbox{for} \,\,\, 1<p<N. \\
	\mbox{}\\
	ii) \,\, \Phi(t)=|t|^{p}+|t|^{q} \,\,\, \mbox{for} \,\,\, 1<p<q<N \,\,\, \mbox{and} \,\,\, q \in (p,p^{*}) \,\,\, \mbox{with} \,\,\, \displaystyle p^{*}=\frac{Np}{N-p}.  \\
	\mbox{}\\
	iii) \,\, \Phi(t)=(1+|t|^2)^{\gamma}-1 \,\,\, \mbox{for} \,\,\, \gamma \in (1, \frac{N}{N-2} ). \\
	\mbox{}\\
	iv) \,\,  \Phi(t)=|t|^{p}ln(1+|t|) \,\,\, \mbox{for} \,\,\, 1<p_0<p<N-1 \,\,\, \mbox{with} \,\,\, \displaystyle p_0=\frac{-1+\sqrt{1+4N}}{2}.
\end{array}
$$\\
For the case $l=1$, a typical example is 
$$
\Phi(t)=|t|\log(1+|t|), \quad \forall t \in \mathbb{R}.
$$
Here we would like to point out that if $1<l\leq m<+\infty$, then function $\Phi$ and its complementary function $\widetilde{\Phi}$ given by 
\[
\widetilde{\Phi}(s) = \max_{t\geq 0}\{ st - \Phi(t)\}, \quad  \mbox{for} \quad s\geq0,
\]
satisfy the $\Delta_{2}$-condition. It is well known in the literature that $W^{1,\Phi}(\mathbb{R}^N)$ is reflexive when $\Phi$ and $\widetilde{\Phi}$ satisfy the $\Delta_{2}$-condition. Thus, in our paper the space $W^{1,\Phi}(\mathbb{R}^N)$ can be nonreflexive, because we are also considering the case $l=1$.

In what follows, the continuous function $f:\mathbb{R} \to \mathbb{R}$ satisfies the following assumptions:

$$
\lim_{t \to 0}\frac{f(t)}{t\phi(t)}=0. \leqno{(f_1)}
$$
$$
\limsup_{t \to +\infty}\frac{|f(t)|}{|t|b(|t|)}<+\infty \leqno{(f_2)}
$$
where $b:(0,+\infty) \to \mathbb{R}$ is continuous function satisfying: \\
$$
0<b_1\leq \frac {t^2b(t)}{B(t)}\leq b_2,~t>0,  \leqno{(b_1)}
$$
where $m<b_1<b_2<l^*$ and $B(t)=\int_{0}^{|t|}sb(s)\,ds$ for all $ t\in \mathbb{R}$ is a N-function. 

There is $\theta>m$ such that 
$$
0<\theta F(t) \leq f(t)t, \quad \forall t \in \mathbb{R} \setminus \{0\}. \leqno{(f_3)}
$$

Finally, related to the potential $V:\mathbb{R}^{N}\rightarrow \mathbb{R}$, we assume that it is a continuous $\mathbb{Z}^N$-perifunction verifying  
$$
0<V_0 =\inf_{x \in \mathbb{R}^{N}}V(x). \eqno{(V_1)}
$$

In what follows, we say that $u \in W^{1,\Phi}(\mathbb{R}^N)$ is a weak solution to $(P)$ whenever
$$
\int_{\mathbb{R}^N}\phi(|\nabla u|)\nabla u \nabla v\,dx+\int_{\mathbb{R}^N}V(x)\phi(|u|)u v\,dx=\int_{\mathbb{R}^N}f(u)v\,dx, \quad \forall v \in W^{1,\Phi}(\mathbb{R}^N).
$$

Under these conditions our main result involving the existence of nontrivial solution to $(P)$ is the following:

\begin{theorem} \label{T1} Assume $(\phi_1) - (\phi_3), (f_1),(f_2),(V_1)$ and $(b_1)$. Then, problem $(P)$ has a nontrivial solution.
\end{theorem}

The theorem above complements to study made in \cite[Theorem 1.5]{AGJ}, because in that paper $\Phi$ and $\widetilde{\Phi}$ satisfy the $\Delta_2$-condition, while that in our paper, it is not necessary $\widetilde{\Phi}$ to satisfy the $\Delta_2$-condition. 

Quasilinear elliptic problems have been considered using different assumptions on the N-function $\Phi$. For example in the papers  \cite{JVMLED},\cite{MSGC},\cite{chung},\cite{MR1},\cite{MR2},\cite{MRep},\cite{MT},\cite{MugnaiPapageorgiou},\cite{MR3} and \cite{fang}, the authors assumed that $\Phi$ and $\widetilde{\Phi}$ satisfy the $\Delta_{2}$-condition, then in those papers the Orlicz-Sobolev space $W^{1,\Phi}(\Omega)$ is a reflexive Banach space. This assertion is used several times in order to get a nontrivial solution for elliptic problems taking into account the weak topology. In the present  paper, the main goal is the use of techniques that allow one to deal with problem $(P)$ without assuming the $\Delta_{2}$-condition in one of the functions $\Phi$ or $\widetilde{\Phi}$. This type of problem brings us many difficulties when we intend to apply variational methods directly in $W^{1,\Phi}(\R^N)$. For example, in our paper the energy functional associated with the problem may be not $C^{1}$, then the classical variational cannot be used here. In order to overcome this difficulty we have used the minimax theory developed by Szulkin \cite{Szulkin}. Moreover, another important difficulty is associated with the fact that we cannot use the weak topology on the space $W^{1,\Phi}(\R^N)$, because it can be nonreflexive.

In recent years many researchers have studied the non-reflexive case. For example, in \cite{GKMS}, Garc\'ia-Huidobro, Khoi, Man\'asevich and Schmitt have considered existence of solution for the following nonlinear eigenvalue problem
\begin{equation} \label{P1} 
\left\{
\begin{array}{ll}
-\Delta_{\Phi}{u}=\lambda \Psi (u), \quad \mbox{in} \quad \Omega \\
u=0, \quad \mbox{on} \quad \partial \Omega,
\end{array}
\right. 
\end{equation}
where $\Omega$ is a bounded domain, $\Phi:\mathbb{R} \to \mathbb{R}$ is a N-function and $\Psi:\mathbb{R} \to \mathbb{R}$ is a continuous function verifying some others technical conditions. In that paper, the authors have studied the situation where $\Phi$ does not satisfy the well known $\Delta_2$-condition. More precisely, in the first part of that paper the authors consider the function
\begin{equation} \label{phi*}
\Phi(t)= (e^{t^{2}}-1)/2, \quad \forall t \in \mathbb{R}.
\end{equation}
More recently, Bocea and  Mih\u{a}ilescu \cite{BM}  made a careful study about the eigenvalues of the problem
\begin{equation} \label{P2} 
\left\{
\begin{array}{ll}
-div(e^{|\nabla u|^{2}}\nabla u)-\Delta u=\lambda u, \quad \mbox{in} \quad \Omega \\
u=0, \quad \mbox{on} \quad \partial \Omega.
\end{array}
\right. 
\end{equation}
After that, Silva, Gon\c calves and Silva \cite{EGS} considered existence of multiple solutions for a class of problem like (\ref{P1}). In that paper the $\Delta_2$-condition is not also assumed and the main tool used was the truncation of the nonlinearity together with a minimization procedure for the energy functional associated to the quasilinear elliptic problem (\ref{P1}).

In \cite{DMKV}, Silva, Carvalho, Silva and Gon\c calves study a class of problem (\ref{P1}) where the energy functional satisfies the mountain pass geometry and the N-function $\widetilde{\Phi}$ does not satisfies the $\Delta_2$-condition and has a polynomial growth. Still related to the mountain geometry, in \cite{AEM}, Alves, Silva and Pimenta also considered the problem (\ref{P1}) for a large class of function $\Psi$, but supposing that the N-function $\Phi$ has an exponential growth like (\ref{phi*}).

\vspace{0.5 cm}

\section{Basics on Orlicz-Sobolev spaces}

In this section we recall some properties of Orlicz and Orlicz-Sobolev spaces, which can be found in \cite{Adams, RR}. First of all, we recall that a continuous function $\Phi : \mathbb{R} \rightarrow [0,+\infty)$ is a
N-function if:
\begin{itemize}
	\item[$(i)$] $\Phi$ is convex.
	\item[$(ii)$] $\Phi(t) = 0 \Leftrightarrow t = 0 $.
	\item[$(iii)$] $\displaystyle\lim_{t\rightarrow0}\frac{\Phi(t)}{t}=0$ and $\displaystyle\lim_{t\rightarrow+\infty}\frac{\Phi(t)}{t}= +\infty$ .
	\item[$(iv)$] $\Phi$ is even.
\end{itemize}
We say that a N-function $\Phi$ verifies the $\Delta_{2}$-condition, if
\[
\Phi(2t) \leq K\Phi(t),\quad \forall t\geq 0,
\]
for some constant $K > 0$. For instance, it can be shown that $\Phi(t)=|t|^{p}/p$ for $p>1$ satisfies the $\Delta_2$-condition, while $\Phi(t)=(e^{t^{2}}-1)/2$ does not satisfy it.

In what follows, fixed an open set $\Omega \subset \mathbb{R}^{N}$ and a N-function $\Phi$, we define the Orlicz space associated with $\Phi$ as
\[
L^{\Phi}(\Omega) = \left\{  u \in L_{loc}^{1}(\Omega) \colon \ \int_{\Omega} \Phi\left(\frac{|u|}{\alpha}\right)dx < + \infty \ \ \mbox{for some}\ \ \alpha >0 \right\}.
\]
The space $L^{\Phi}(\Omega)$ is a Banach space endowed with the Luxemburg norm given by
\[
\Vert u \Vert_{\Phi} = \inf\left\{  \alpha > 0 : \int_{\Omega}\Phi\Big(\frac{|u|}{\alpha}\Big)dx \leq1\right\}.
\]
The complementary function $\widetilde{\Phi}$ associated with $\Phi$ is given
by its Legendre's transformation, that is,
\[
\widetilde{\Phi}(s) = \max_{t\geq 0}\{ st - \Phi(t)\}, \quad  \mbox{for} \quad s\geq0.
\]
The functions $\Phi$ and $\widetilde{\Phi}$ are complementary each other and satisfy the inequality below 
\begin{equation} \label{DES2D}
	\widetilde{\Phi}(\Phi'(t))\leq \Phi(2t), \quad \forall t >0.
\end{equation}
Moreover, we also have a Young type inequality given by
\begin{equation} \label{Young}
st \leq \Phi(t) + \widetilde{\Phi}(s), \quad \forall s, t\geq0.
\end{equation} 

Using the above inequality, it is possible to prove a H\"older type inequality, that is,
\[
\Big| \int_{\Omega}uvdx \Big| \leq 2 \Vert u \Vert_{\Phi}\Vert v \Vert_{ \widetilde{\Phi}},\quad \forall u \in L^{\Phi}(\Omega) \quad \mbox{and} \quad \forall v \in L^{\widetilde{\Phi}}(\Omega).
\]

Another important function related to function $\Phi$, it is the Sobolev conjugate function $\Phi_{*}$ of $\Phi$ defined by 
$$
\Phi^{-1}_{*}(t)=\displaystyle\int^{t}_{0}\displaystyle\frac{\Phi^{-1}(s)}{s^{(N+1)/N}} ds \,\,\, \mbox{for} \,\,\, t >0,
$$
when
$$
\displaystyle\int^{+\infty}_{1}\displaystyle\frac{\Phi^{-1}(s)}{s^{(N+1)/N}} ds=+\infty.
$$

The corresponding Orlicz-Sobolev space is defined by
\[
W^{1, \Phi}(\Omega) = \Big\{ u \in L^{\Phi}(\Omega) \ :\ \frac{\partial u}{\partial x_{i}} \in L^{\Phi}(\Omega), \quad i = 1, ..., N\Big\},
\]
endowed with the norm
\[
\Vert u \Vert = \Vert \nabla u \Vert_{\Phi} + \Vert u \Vert_{\Phi}.
\]

The space $W_0^{1,\Phi}(\Omega)$ is defined as the weak$^*$ closure of $C_0^{\infty}(\Omega)$ in $W^{1,\Phi}(\Omega)$.  Here we refer the readers to the important works \cite{gossez,gossez2}. The spaces $L^{\Phi}(\Omega)$, $W^{1, \Phi}(\Omega)$ and $W_0^{1, \Phi}(\Omega)$ are separable and reflexive, when $\Phi$ and $\widetilde{\Phi}$ satisfy the $\Delta_{2}$-condition.

If $|\Omega|<+\infty$, $E^{\Phi}(\Omega)$ denotes the closure of $L^{\infty}(\Omega)$ in $L^{\Phi}(\Omega)$ with respect to the norm $\|\,\,\|_{\Phi}$. When $|\Omega|=+\infty$, $E^{\Phi}(\Omega)$ denotes the closure of $C_0^{\infty}(\Omega)$ in $L^{\Phi}(\Omega)$ with respect to the norm $\|\,\,\|_{\Phi}$. In any one of these cases, $L^{\Phi}(\Omega)$ is the dual space of $E^{\widetilde{\Phi}}(\Omega)$, while $L^{\widetilde{\Phi}}(\Omega)$ is the dual space of $E^{\Phi}(\Omega)$. Moreover, $E^{\Phi}(\Omega)$ and $E^{\widetilde{\Phi}}(\Omega)$ are separable spaces and any continuous linear functional $M:E^{\Phi}(\Omega) \to \mathbb{R}$ is of the form
$$
M(v)=\int_{\Omega}v(x)g(x)\,dx \quad \mbox{for some} \quad g \in L^{\widetilde{\Phi}}(\Omega).
$$
We recall that if $\Phi$ verifies $\Delta_2$-condition, we then have $E^{\Phi}(\Omega)=L^{\Phi}(\Omega)$.  

$W^{1}E^{\Phi}(\Omega)$ is defined analogously and it is also separable. Moreover, the Banach space $W_0^{1}E^{\Phi}(\Omega)$ is the closure of $C_0^{\infty}(\Omega)$ in $W^{1,\Phi}(\Omega)$ with respect to the norm $\Vert \,\,\, \Vert$.

Next we state a very important lemma due to Donaldson \cite[Proposition 1.1]{donaldson} that will be use later on.

\begin{lemma}  \label{Estrela} Assume that $\Phi$ is a N-function.  If $(u_n) \subset W^{1,\Phi}(\Omega) $ is a bounded sequence, then there are a subsequence of $(u_n)$, still denoted by itself, and $u \in  W^{1,\Phi}(\Omega)$ such that
	$$
	u_n \stackrel{*}{\rightharpoonup} u \quad \mbox{in} \quad W^{1,\Phi}(\Omega)
	$$
	and
	$$
	\int_{\Omega}u_n v \,dx  \to  \int_{\Omega}u v \,dx, \quad  \int_{\Omega}\frac{\partial u_n}{\partial x_i} w \,dx  \to  \int_{\Omega}\frac{\partial u}{\partial x_i} w \,dx, \quad \forall v,w \in  E^{\widetilde{\Phi}}(\Omega).
	$$
\end{lemma}	

As an immediate consequence of the last lemma is the  following result that applies an important role in our work.  
\begin{corollary}  \label{Corolario0} Assume that $\Phi$ is a N-function. If $(u_n) \subset W^{1,\Phi}(\Omega) $ is a bounded sequence with 
	$u_n \to u$ in $L_{loc}^{\Phi}(\Omega)$, then $u \in  W^{1,\Phi}(\Omega)$. 	
\end{corollary}

The lemma just above is crucial when the space $W^{1,\Phi}(\Omega)$ is not reflexive, for example if $\Phi(t)=|t|\log(1+|t|)$. However, if $\Phi(t)=|t|^{p}/p$ and $p>1$, the above lemma is not necessary since $\Phi$ and $\tilde{\Phi}$ satisfy the $\Delta_2$-condition, and so, $W^{1,\Phi}(\Omega)$ is reflexive. Here we would like to point out that the condition $(\phi_3)$ ensures that ${\Phi}$ and $\tilde{\Phi}$ verify the $\Delta_2$-condition when $l>1$, for more details see Fukagai and Narukawa \cite{FN}.

 In  \cite{FN}, it is proved the following result

\begin{lemma}\label{F1} Assume that $(\phi_1)-(\phi_3)$ hold and let $\xi_{0}(t)=\min\{t^{l},t^{m}\}$,\linebreak $ \xi_{1}(t)=\max\{t^{l},t^{m}\},$ for all $t\geq 0$. Then,
	$$
	\xi_{0}(\rho)\Phi(t) \leq \Phi(\rho t) \leq  \xi_{1}(\rho)\Phi(t) \;\;\; \mbox{for} \;\; \rho, t \geq 0
	$$
	and
	$$
	\xi_{0}(\|u\|_{\Phi}) \leq \int_{\mathbb{R}^{N}}\Phi(|u|)\,dx \leq \xi_{1}(\|u\|_{\Phi})  \;\;\; \mbox{for} \;\; u \in L^{\Phi}(\mathbb{R}^{N}).
	$$
\end{lemma}

\begin{lemma} \label{F2} Assume that $(\phi_1)-(\phi_3)$ hold and let  $\xi_{2}(t)=\min\{t^{l^{*}},t^{m^{*}}\},$ $\xi_{3}(t)=\max\{t^{l^{*}},t^{m^{*}}\}$, for all $t\geq 0$. Then,
	$$
	\xi_{2}(\rho)\Phi_*(t) \leq \Phi_*(\rho t) \leq \xi_{3}(\rho)\Phi_*(t) \;\;\; \mbox{for} \;\; \rho, t \geq 0
	$$
	and
	$$
	\xi_{2}(\|u\|_{\Phi_*}) \leq \int_{\mathbb{R}^{N}}\Phi_*(|u|)\,dx \leq \xi_{3}(\|u\|_{\Phi_*})  \;\;\; \mbox{for} \;\; u \in L_{\Phi_*}(\mathbb{R}^{N}).
	$$
\end{lemma}

The next lemma is a technical result that will be used later on. It will be important in our approach, because we are only supposing that $\Phi$ is a N-function.   

\begin{lemma} \label{pontual} ( Almost weak converge in $L^{\Phi}(\mathbb{R}^N)$) Let $(w_n) \subset L^{\Phi}(\mathbb{R}^N)$ be a bounded sequence with $w_n(x) \to w(x)$ a.e. in $\mathbb{R}^N$. Then, $w \in  L^{\Phi}(\mathbb{R}^N)$ and 
	$$
	\int_{\mathbb{R}^N}w_n v \, dx \to \int_{\mathbb{R}^N}w v \, dx, \quad \forall v \in C_0^{\infty}(\mathbb{R}^N). 
	$$	
\end{lemma}
\begin{proof} To begin with, we will prove that $w \in L^{\Phi}(\mathbb{R}^N)$. If $\|w_n\|_{L^{\Phi}(\mathbb{R}^N)} \to 0$, we have  $w_n \to 0$ in $L^{\Phi}(\mathbb{R}^N)$, then  $w=0$, finishing the proof. 
	
	In what follows, we will assume that $\|w_n\|_{L^{\Phi}(\mathbb{R}^N)} \not\to 0$, consequently for some subsequence, still denoted by $(w_n)$,  
	$$
	\|w_n\|_{L^{\Phi}(\mathbb{R}^N)} \geq \delta, \quad \forall n \in\mathbb{N}, 
	$$
	and 
	$$
	\|w_n\|_{L^{\Phi}(\mathbb{R}^N)} \to \alpha>0.
	$$
	Since 
	$$
	\int_{\mathbb{R}^N}\Phi\left(\frac{|w_n|}{\|w_n\|_{L^{\Phi}(\mathbb{R}^N)}}\right)\,dx \leq 1, \quad \forall n \in \mathbb{N},
	$$
	the Fatou's Lemma leads to
	$$
	\int_{\mathbb{R}^N}\Phi\left(\frac{|w|}{\alpha}\right)\,dx \leq 1, 
	$$
	from where it follows that $w \in L^{\Phi}(\mathbb{R}^N)$.

	Now, for a fixed $v \in  C_0^{\infty}(\mathbb{R}^N)$, we set $\Omega=supp(v)$ and for $k \in \mathbb{N}$ 
	$$
	\Omega_k=\{x \in \Omega\,:\, \forall n \geq k, \,\, \, |w_n(x)-w(x)|\leq 1  \}.
	$$
	Since $w_n(x) \to w(x)$ a.e. in $\mathbb{R}^N$, a simple computation gives
	$$
	mes(\Omega_k) \to mes(\Omega)  \quad \mbox{and} \quad  mes(\Omega \setminus \Omega_k) \to 0 \quad \mbox{as} \quad k \to +\infty. 
	$$
	Given $\epsilon>0$, let us fix $k$ such that $\|v\|_{L^{\widetilde{\Phi}}(\Omega \setminus \Omega_k)}<\frac{\epsilon}{4M}$, where 
	$$
	M=\max\left\{\sup_{n \in \mathbb{N}}\|w_n\|_{L^{\widetilde{\Phi}}(\Omega)},\|w\|_{L^{\widetilde{\Phi}}(\Omega)}\right\}.
	$$	
	Using this information, we find
	$$
	\left| \int_{\Omega}w_n v \, dx - \int_{\Omega}w v \, dx \right| \leq \int_{\Omega_k}|w_n -u||v|\,dx + \frac{\epsilon}{2}, \quad \forall n \in \mathbb{N}.
	$$
	By definition of $\Omega_k$, for $n \geq k$ we have
	$$
	|w_n(x)-w(x)|\leq 1, \quad \forall x \in \Omega_k.
	$$
	Hence by Lebesgue dominated convergence theorem
	$$
	\lim_{n \to +\infty}\int_{\Omega_k}|w_n -w||v|\,dx=0.
	$$
	Thus, there is $n_0=n_0(\epsilon,k) \in\mathbb{N}$ such that 
	$$
	\left| \int_{\Omega}w_n v \, dx - \int_{\Omega}w v \, dx \right| < \epsilon, \quad \forall n \geq n_0,
	$$
	as asserted. 
		
\end{proof}

Before concluding this section, we are going to prove Theorem \ref{Lions1}.
\begin{proof}
	Since
	$$
	\lim_{t \to 0}\frac{B(|t|)}{\Phi(|t|)}=0,
	$$
	given $\tau>0$, there is $\epsilon>0$ such that 
	$$
	B(|t|)| \leq \frac{\tau}{3 M} \Phi(|t|), \quad \forall t \in [-\epsilon, \epsilon], 
	$$
	where $M:=\sup_n\int_{\mathbb{R}^N}\Phi(|w_n|)dx$. Moreover, as 
	$$
	\lim_{t \to \infty}\frac{B(|t|)}{\Phi_*(|t|)}=0,
	$$
	there is $T>0$ such that
	$$
	B(|t|)| \leq \frac{\tau}{3 M^*} \Phi_*(|t|), \quad \forall |t|> T, 
	$$
	where $M^*:=\sup_n\int_{\mathbb{R}^N}\Phi_*(|w_n|)dx$. Therefore,
	$$
	\int_{\mathbb{R}^N}B(|w_n|)\,dx \leq\left(\int_{[|w_n|\leq \epsilon]}+\int_{[\epsilon<|w_n|\leq T]}+\int_{[|w_n|> T]}\right)B(|w_n|)\,dx \leq \frac{2\tau}{3}+ B(T)mes([|w_n|>\epsilon]).
	$$
	Now, the theorem follows using the fact that $\tau$ is arbitrary and that $mes([|w_n|>\epsilon]) \to 0$.
\end{proof}

\section{Proof of Theorem \ref{lieb}}

The main goal of this section is to show a Lieb type result for a large class of Orlicz-Sobolev spaces, without assuming the $\Delta_2$-condition. A version of Lieb type result for Sobolev space can be found in Kavian \cite[6.2 Lemme]{Kavian}. 

The first lemma this section is a technical result that is a key point in the proof of the Lieb type result for Orlicz-Sobolev spaces.

\begin{lemma} \label{L1} Let $\Phi \in C^{1}(\mathbb{R},[0,+\infty))$ be a $N$-function and $u \in W^{1,\Phi}(\mathbb{R}^N) \setminus \{0\}$ such that $\int_{\mathbb{R}^N}\Phi(|\nabla u|)\,dx \leq M$. Then, there is $y_0 \in \mathbb{R}^N$ that depends on $u$ and $C_0>0$ that does not depend on $u$ and $y_0 \in \mathbb{R}^N$  such that 
	$$
	\left(2+M\left(\int_{\mathbb{R}^N}\Phi(|u/2|)\,dx\right)^{-1}\right)^{N}mes[K(y_0) \cap supp(u)] \geq C_0,
	$$
	where $K(z)=\displaystyle \prod_{i=1}^{N}\left(z_i-\frac{1}{2},z_i+\frac{1}{2}\right)$ for all $z \in \mathbb{R}^N$. 
\end{lemma}
\begin{proof} First of all we claim that there is $y_0 \in \mathbb{R}^N$ such that 
	\begin{equation} \label{0E0}
	\int_{\mathbb{R}^N}\Phi(|\nabla u|) \chi_{K(y_0)}\,dx < \left(1+M\left(\int_{\mathbb{R}^N}\Phi(|u/2|)\,dx\right)^{-1}\right)\int_{\mathbb{R}^N}\Phi(|u/2|)\chi_{K(y_0)}\,dx ,
	\end{equation}
	where $\chi_{K(y_0)}$ is the characteristic function associated with the set $K(y_0)$.  
	
	Otherwise, we must have 
	$$
	M \geq \int_{\mathbb{R}^N}\Phi(|\nabla u|)\,dx \geq \left(1+M\left(\int_{\mathbb{R}^N}\Phi(|u/2|)\,dx\right)^{-1}\right)\int_{\mathbb{R}^N}\Phi(|u/2|)\,dx > M,
	$$	
	which is impossible. 
	
	\begin{claim}$\Phi(|u/2|) \in W^{1,1}(K(y_0))$.
		
	\end{claim}
	
	Indeed, since $\Phi$ is increasing 
	\begin{equation} \label{0E1}
	\int_{K(y_0)}\Phi(|u/2|)\,dx \leq \int_{K(y_0)}\Phi(|u|)\,dx<+\infty.
	\end{equation}
	On the other hand, 
	$$
	\int_{K(y_0)}|\nabla \Phi(|u/2|)|\,dx=\frac{1}{2}\int_{K(y_0)}\Phi'(|u/2|){|\nabla u|}\,dx.
	$$
	By Young's inequality 
	$$
	\int_{K(y_0)}|\nabla \Phi(|u/2|)|\,dx \leq \frac{1}{2}\int_{K(y_0)}\Phi(|\nabla u|)\,dx+\frac{1}{2}\int_{K(y_0)}\widetilde{\Phi}(\Phi'(|u/2|))\,dx.
	$$
	Recalling that 
	$$
	\widetilde{\Phi}(\Phi'(t))\leq \Phi(2t), \quad \forall t >0,
	$$
	we get
	\begin{equation} \label{0E2}
	\int_{K(y_0)}|\nabla (\Phi(|u/2|))|\,dx \leq \frac{1}{2}\int_{K(y_0)}\Phi(|\nabla u|)\,dx+\frac{1}{2}\int_{K(y_0)}\Phi(|u|)\,dx.
	\end{equation} 
	The claim follows from (\ref{0E1}) and (\ref{0E2}). 
	
	Now, using the continuous Sobolev  embedding $W^{1,1}(K(y_0)) \hookrightarrow L^{1^*}(K(y_0))$ where  $1^*=\frac{N}{N-1}$, there is $C_1>0$ such that 
	$$
	C_1\|w\|_{L^{1^*}(K(y_0))}\,dx \leq \int_{K(y_0)}(|\nabla w|+|w|)\,dx, \quad \forall w \in W^{1,1}(K(y_0)).
	$$
	Hence
	\begin{equation} \label{00E2}
	C_1\left(\int_{K(y_0)}|\Phi(|u/2|)|^{1^*}\,dx\right)^{\frac{1}{1^*}} \leq \int_{K(y_0)}(|\nabla (\Phi(|u/2|))|+|\Phi(|u/2|)|)\,dx, \quad \forall u \in W^{1,1}(\mathbb{R}^N).
	\end{equation}
	From (\ref{0E0})-(\ref{00E2}),
	$$
	C_1\left(\int_{K(y_0)}|\Phi({u}/{2})|^{1^*}\,dx\right)^{\frac{1}{1^*}} \leq \left(2+M\left(\int_{\mathbb{R}^N}\Phi(|u/2|)\,dx\right)^{-1}\right)\int_{K(y_0)}\Phi(|u/2|)\, dx,
	$$	
	leading to 
	$$
	C_1\leq \left(2+M\left(\int_{\mathbb{R}^N}\Phi(|u/2|)\,dx\right)^{-1}\right)mes[K(y_0) \cap supp(u)]^{\frac{1}{N}},
	$$
	that is, 
	$$
	C_0\leq \left(2+M\left(\int_{\mathbb{R}^N}\Phi(|u/2|)\,dx\right)^{-1}\right)^{N}mes[K(y_0) \cap supp(u)].
	$$
\end{proof}

Now, we are ready to prove Theorem \ref{lieb}.

\begin{proof} To begin with, we will apply Lemma \ref{L1} for function $\left(|u_n|-\frac{\epsilon}{2}\right)^{+}$. Note that 
	$$
	\int_{\mathbb{R}^N}\Phi\left(\frac{1}{2}\left(|u_n|-\frac{\epsilon}{2}\right)^{+}\right)\,dx \geq \int_{|u_n|>\epsilon}\Phi\left(\frac{1}{2}\left(|u_n|-\frac{\epsilon}{2}\right)^{+}\right)\,dx \geq \Phi\left(\frac{\epsilon}{4}\right)mes[|u_n|>\epsilon] \geq \Phi\left(\frac{\epsilon}{4}\right)\delta,
	$$	
	from where it follows that
	$$
	\left(\int_{\mathbb{R}^N}\Phi\left(\frac{1}{2}\left(|u_n|-\frac{\epsilon}{2}\right)^{+}\right)\,dx\right)^{-1}\leq \frac{1}{\Phi\left(\frac{\epsilon}{4}\right)\delta}.
	$$
	Since 
	$$
	C_0\leq \left(2+M\left(\int_{\mathbb{R}^N}\Phi\left(\frac{1}{2}\left(|u_n|-\frac{\epsilon}{2}\right)^{+}\right)\,dx\right)^{-1}\right)^{N}mes \left[K(y_n) \cap supp\left(|u_n|-\frac{\epsilon}{2}\right)^{+}\right],	
	$$
	we get
	$$
	C_0\leq \left(2+M\frac{1}{\Phi\left(\frac{\epsilon}{4}\right)\delta}\right)^{N}mes \left[K(y_n) \cap supp\left(|u_n|-\frac{\epsilon}{2}\right)^{+}\right].
	$$
	On the other hand, as $supp\left(|u_n|-\frac{\epsilon}{2}\right)^{+}=[|u_n| \geq \frac{\epsilon}{2}]$, we derive that
	$$
	mes[K(y_n) \cap [|u_n|\geq \frac{\epsilon}{2}]] \geq C_2, \quad \forall n \in \mathbb{N},
	$$
	for some $C_2>0$. Now, using the fact that there is $z_n \in \mathbb{Z}^N$ such  that 
$$
K(y_n) \subset \tilde{K}(z_n)=\displaystyle \prod_{i=1}^{N}\left(z_{n,i}-\frac{3}{2},z_{n,i}+\frac{3}{2}\right),
$$
we obtain	
$$
	\int_{\tilde{K}(0)}\Phi(|v_n|)\,dx\geq \int_{K(y_n) \cap [|u_n|\geq \frac{\epsilon}{2}]}\Phi(|u_n|)\,dx \geq \Phi\left(\frac{\epsilon}{4}\right) mes[K(y_n) \cap [|u_n|\geq \frac{\epsilon}{2}]] 
	$$
	that is,
	$$
	\int_{\tilde{K}(0)}\Phi(|v_n|)\,dx\geq \Phi\left(\frac{\epsilon}{4}\right) C_2=C_3>0, \quad \forall n \in \mathbb{N}.
	$$
As $(v_n)$ is bounded, the compact embedding $W^{1,\Phi}(\mathbb{R}^N) \to L^{\Phi}(\tilde{K}(0))$ ensures that $v_n \to v$ in $L^{\Phi}(\tilde{K}(0))$ for some subsequence. Thus,  
	$$
	\int_{\tilde{K}(0)}\Phi(|v|)\,dx\geq C_3>0, 
	$$
	showing that $v \not= 0$, as asserted.
\end{proof}

\section{Technical results}

The energy functional $I: W^{1,\Phi}(\mathbb{R}^N) \to \mathbb{R} $ associated with $(P)$ given by
\begin{equation} \label{I}
	I(u)=\int_{\mathbb{R}^N}\Phi(|\nabla u|)dx+\int_{\mathbb{R}^N}V(x)\Phi(| u|)dx-\int_{\mathbb{R}^N}F(u) dx
\end{equation}
is well defined. The functional $\mathcal{F}: W^{1,\Phi}(\mathbb{R}^N) \to \mathbb{R} $ given by
$$
\mathcal{F}(u)=\int_{\mathbb{R}^N}F(u) dx
$$
belongs to $C^{1}( W^{1,\Phi}(\mathbb{R}^N),\mathbb{R})$ with
$$
\mathcal{F}'(u)v=\int_{\mathbb{R}^N}f(u)v \,dx, \quad \forall u,v \in  W^{1,\Phi}(\mathbb{R}^N).
$$

Related to the functional $Q: W^{1,\Phi}(\mathbb{R}^N) \to \mathbb{R} $ given by
\begin{equation} \label{Q}
	Q(u)=\int_{\mathbb{R}^N}\Phi(|\nabla u|)dx+\int_{\mathbb{R}^N}V(x)\Phi(| u|)dx,
\end{equation}
we know that it is strictly convex   and  l.s.c. with respect to the weak$^*$ topology. Moreover, $Q \in C^{1}(W^{1,\Phi}(\mathbb{R}^N),\mathbb{R})$ when $\Phi$ and $\widetilde\Phi$ satisfy the  $\Delta_2$-condition. Therefore, in our case when $l=1$, we cannot guarantee that $\widetilde\Phi$ satisfies the  $\Delta_2$-condition. From this, the functional $I$ is not $C^1$ when $\ell=1$.  

From the above commentaries, in the present paper we will use a minimax method developed by Szulkin  \cite{Szulkin}. In this sense, we will say that $u \in W^{1,\Phi}(\mathbb{R}^N)$ is a critical point for $I$ if $0 \in \partial I(u) = \partial Q(u) - \mathcal{F}'(u)$. Then $u \in  W^{1,\Phi}(\mathbb{R}^N)$ is a critical point of $I$ if, and only if, $\mathcal{F}'(u) \in \partial Q(u)$, what, since $Q$ is convex, is equivalent to
\begin{equation} \label{E1}
	Q(v)-Q(u)\geq \int_{\mathbb{R}^N}f(u)(v-u)\,dx, \quad \forall v \in   W^{1,\Phi}(\mathbb{R}^N).
\end{equation}
\begin{lemma} \label{solucao} If $u \in W^{1,\Phi}(\mathbb{R}^N)$ is a critical point of $I$ in $W^{1,\Phi}(\mathbb{R}^N)$, then $u$ is a weak solution to $(P)$.
	
\end{lemma}

\begin{proof} Since $\Phi$ satisfies the $\Delta_2$-condition, we claim that $Q$ is G\^ateaux differentiable, that is, $\frac{\partial Q(u)}{\partial v}$ exists for all  $u,v \in W^{1,\Phi}(\mathbb{R}^N)$ with 
\begin{equation} \label{Q}
\frac{\partial Q(u)}{\partial v}=\int_{\mathbb{R}^N}\phi(|\nabla u|)\nabla u \nabla v\,dx+\int_{\mathbb{R}^N}V(x)\phi(| u|)uv\,dx.
\end{equation} 	
Indeed, for each $v \in W^{1,\Phi}(\mathbb{R}^N)$ and $t\in [-1,1] \setminus \{0\}$,
$$
\Phi(|\nabla u +t\nabla v|)-\Phi(|\nabla u|)=t\phi(|\nabla u+st\nabla v|)(\nabla u+st\nabla v)\nabla v,
$$
for some $s \in (0,1)$. Consequently,
$$
\left| \frac{\Phi(|\nabla u +t\nabla v|)-\Phi(|\nabla u|)}{t}\right|=\phi(|\nabla u+st\nabla v|)|\nabla u+st\nabla v||\nabla v|.
$$
Since $\Phi$ satisfies the $\Delta_2$-condition, by Young inequality (\ref{Young}) and (\ref{DES2D}), there is $C>0$ such that 
$$
\phi(|\nabla u+st\nabla v|)|\nabla u+st\nabla v||\nabla v| \leq C\Phi(|\nabla u|+|\nabla v| ) +\Phi(|\nabla v|) \in L^{1}(\mathbb{R}^N)
$$
A similar argument works to show that 
$$
\phi(|u+st v|)| u+st v|| v| \leq C\Phi(|u|+| v| ) +\Phi(|v|) \in L^{1}(\mathbb{R}^N).
$$
Now, by using Lebesgue dominated convergence theorem,  we derive that 
$$
\lim_{t \to 0}\frac{Q(u+tv)-Q(u)}{t}=\int_{\mathbb{R}^N}(\phi(|\nabla u|)\nabla u \nabla v+V(x)\phi(|u|)u v)\,dx,
$$
showing (\ref{Q}). 

Recalling that the functional $\mathcal{F}: W^{1,\Phi}(\mathbb{R}^N) \to \mathbb{R} $ given by
$$
\mathcal{F}(u)=\int_{\mathbb{R}^N}F(u) dx
$$
belongs to $C^{1}( W^{1,\Phi}(\mathbb{R}^N),\mathbb{R})$ with 
$$
\mathcal{F}'(u)v=\int_{\mathbb{R}^N}f(u)v \,dx, \quad \forall u,v \in  W^{1,\Phi}(\mathbb{R}^N),
$$
it follows that $I$ is G\^ateaux differentiable with 
\begin{equation} \label{E1.0}
		\frac{\partial I(u)}{\partial v}=\frac{\partial Q(u)}{\partial v}-\mathcal{F}'(u)v, \quad \forall u,v \in  W^{1,\Phi}(\mathbb{R}^N),
\end{equation}
	or equivalently
	$$
	\frac{\partial I(u)}{\partial v}=\int_{\mathbb{R^N}}\phi(|\nabla u|)\nabla u \nabla v\,dx+	\int_{\mathbb{R}^N}V(x)\phi(| u|)u v\,dx-\int_{\mathbb{R}^N}f(u)v\,dx, \quad \forall u,v \in  W^{1,\Phi}(\mathbb{R}^N).
	$$

	As $u$ is a critical point of $I$, it follows that
	$$
	Q(w) - Q(u) \geq \int_{\R^N}f(u)(w-u)\,dx, \quad \forall v \in  W^{1,\Phi}(\mathbb{R}^N).
	$$	
	Thus, for each $v \in W^{1,\Phi}(\mathbb{R}^N)$ and $t>0 $,  
	$$
	\frac{Q(u+tv) - Q(u) }{t}\,dx  \geq \int_{\R^N}f(u)v\,dx.
	$$	
	Taking the limit when $t \to 0$, we get
	$$
\frac{\partial Q(u)}{\partial v} \geq \int_{\R^N}f(u)v\,dx, \quad \forall v \in W^{1,\Phi}(\mathbb{R}^N),
	$$
	or equivalently
	$$
	\int_{\R^N}\phi(|\nabla u|)\nabla u \nabla v  \, dx+ \int_{\R^N}V(x)\phi(|u|) u v  \, dx-\int_{\R^N}f(u)v\,dx \geq 0, \quad \forall v \in W^{1,\Phi}(\mathbb{R}^N).
	$$
The last inequality ensures that 
$$
\int_{\R^N}\phi(|\nabla u|)\nabla u \nabla v  \, dx + \int_{\R^N}V(x)\phi(|u|) u v  \, dx-\int_{\R^N}f(u)v\,dx = 0, \quad \forall v \in W^{1,\Phi}(\mathbb{R}^N), 
$$
finishing the proof. 
\end{proof}	

\begin{lemma}\label{passo} The functional $I$ satisfies the mountain pass geometry, that is, \\
	(i) There exist $\rho$, $\eta>0$, such that $I(u)\geq
	\eta$, if $\|u\|=\rho$. \\
	
	\noindent (ii) For any $ \psi \in C^{\infty}_{0}(\mathbb{R}^{N}) \setminus \{0\}$, $I(t \psi)\rightarrow - \infty$ as $t\mapsto +\infty$.
\end{lemma}
\begin{proof} \,\, \textbf{(i)} 
 From assumptions $(f_{1})-(f_{3})$, given $\epsilon >0$, there exists $C_\epsilon>0$ such that 
\begin{eqnarray}\label{crtescimentoF}
	0\leq F(t)\leq \displaystyle\frac{\epsilon m}{\theta} \Phi(|t|)+ {C_\epsilon} B(|t|)\,\,\,\,\,\,\,\, \forall t \in \mathbb{R}.
\end{eqnarray}
Thus, 
$$
I(u) \geq \displaystyle \int_{\mathbb{R}^N}\Phi(|\nabla u|)\,dx +\left(V_0 - \frac{\epsilon m}{\theta}\right)
\displaystyle \int_{\mathbb{R}^N}\Phi(|u|)\,dx
- C_\epsilon\displaystyle \int_{\mathbb{R}^N}B(|u|)\,dx
$$
Hence, for $\epsilon$ small enough, 
$$
J(u) \geq C_{1} \displaystyle \left(\xi_{0}(\|\nabla u\|_{\Phi}) +\xi_{0}(\|u\|_{\Phi}) \right)
- C_{2}\xi_{3}(\|u\|_{B}).
$$
Choosing $\rho >0$  such that
$$
\|u\|=\|\nabla u\|_{\Phi} + \|u\|_{\Phi} = \rho <1  \ \ \mbox{and} \ \ \|u\|_{B}\leq C(\|\nabla u \|_{\Phi}+ \|u\|_{\Phi}) <\rho <1,
$$
we obtain
$$
I(u) \geq C_{1} \displaystyle (\|\nabla u\|^{m}_{\Phi} +\|u\|^{m}_{\Phi})
- C_{2}\|u\|^{b_1}_{B},
$$
which yields 
$$
I(u) \geq C_{3}\|u\|^{m} - C_4 \|u\|^{b_1},
$$
for some positive constants $C_3$ and $C_4$.  Since $0<m<b_1$, there exists $\eta >0$ such that 
$$ 
I(u)\geq \eta \,\,\, \mbox{ for all} \,\,\, \|u\|=\rho. 
$$
\noindent \textbf{(ii)}  From $(f_{3})$, there exist $C_{5}, C_{6}>0$ such that
$$
F(t)\geq C_{5}|t|^{\theta} - C_{6}, \ \ \mbox{for all} \ \ t \in \mathbb{R}.
$$
Fixing $ \psi \in C^{\infty}_{0}(\mathbb{R}^{N}) \setminus \{0\}$, the last inequality leads to 
$$
I(t \psi) \leq \xi_{1}(t) ( \xi_{1}(\| \nabla \psi \|_{\Phi}) + \|V\|_\infty\xi_{1}(\|  \psi \|_{\Phi}))- C_{5}t^{\theta}\displaystyle\int_{\mathbb{R}^{N}}|\psi|^{\theta}\,dx + C_{6} \mbox{supp} \, \psi.
$$
Thus, for $t$ sufficient large,
$$
I(t \phi) \leq t^{m}( \xi_{1}(\| \nabla \psi \|_{\Phi}) + \|V\|_\infty\xi_{1}(\|  \psi \|_{\Phi}))- C_{5}t^{\theta}\displaystyle\int_{\mathbb{R}^{N}}|\psi|^{\theta}\,dx + C_{6} \mbox{supp}\, \psi.
$$
Since $ m < \theta$, the result follows. \end{proof}

The last lemma permits to apply a version of the Mountain Pass Theorem found in \cite[Theorem 3.1]{AlvesdeMorais} to guarantee the existence of a $(PS)$ sequence $(v_n) \subset W^{1,\Phi}(\mathbb{R}^N)$ for $I$ associated with the mountain pass level of $I$, denoted by $d$, that is, $I(v_n)\to d>0 $ and there is  $\tau_n\to 0$ in $\R$ such that
\begin{equation} \label{sequencia2}
	Q(w)-Q(v_n) \geq \int_{\mathbb{R}^N}f(v_n)(w-v_n)\,dx- \tau_n\|w-v_n\|,
\end{equation}
for all $w \in W^{1,\Phi}(\mathbb{R}^N)$ and  $n \in\mathbb{N}$.

\begin{lemma}\label{doisitens}
Any $(PS)$ sequence for $I$ is bounded. 
\end{lemma}
\begin{proof} 
First of all, we recall that $(u_n)$ is a $(PS)$ sequence for $I$ when there are $c \in \mathbb{R}$ and $\tau_n\to 0$ in $\R$ such that $I(u_n)\to c $ and  
\begin{equation*}
	Q(v)-Q(u_n) \geq \int_{\mathbb{R}^N}f(u_n)(v-u_n)\,dx- \tau_n\|v-u_n\|,
\end{equation*}
for all $v \in W^{1,\Phi}(\mathbb{R}^N)$ and $n \in\mathbb{N}$. As $Q$ is G\^ateaux differentiable, it follows that \linebreak  $\frac{\partial I(u_n)}{\partial u_n}=o_n(1)\|u_n\|$, that is, 
$$
\int_{\mathbb{R^N}}\phi(|\nabla u_n|)|\nabla u_n|^2\,dx+	\int_{\mathbb{R}^N}V(x)\phi(|u_n|)|u_n|^2 \,dx-\int_{\mathbb{R}^N}f(u_n)u_n\,dx=o_n(1)\|u_n\|, \quad \forall n \in \mathbb{N}.
$$	
Hence, there is $C>0$ such that
\begin{eqnarray*}
	C+o_n(1)\|u_{n}\|\geq  I(u_{n})-\frac{1}{\theta}\frac{\partial I(u_n)}{\partial u_n}, \,\,\, \forall n \in \mathbb{N}.
\end{eqnarray*}
From $(f_{3})$,
\begin{eqnarray*}
	C+o_n(1)\|u_{n}\|&\geq &
\min\{1,V_0\}\left(\frac{\theta-m}{\theta}\right)\displaystyle\int_{\mathbb{R}^{N}}(\Phi(|\nabla u_{n}|)+\Phi(|u_{n}|))\,dx \\
	&\geq & \min\{1,V_0\}\left(\frac{\theta-m}{\theta}\right)\biggl[\xi_{0}(\|\nabla u_{n}\|_{\Phi})+ \xi_{0}(\|u_{n}\|_{\Phi})\biggl].
\end{eqnarray*}

\noindent Seeking for a contradiction, we assume that for some subsequence, $\|u_{n}\|\rightarrow +\infty$. This way, we need to study the following situations: \\

\noindent a) $\|\nabla u_{n}\|_{\Phi}\rightarrow +\infty$ and $\|u_{n}\|_{\Phi}\rightarrow +\infty$, \\

\noindent b) $\|\nabla u_{n}\|_{\Phi}\rightarrow +\infty$ and $\|u_{n}\|_{\Phi}$ is bounded, \\

\noindent and \\

\noindent c) $\|\nabla u_{n}\|_{\Phi}$ is bounded and $\|u_{n}\|_{\Phi}\rightarrow +\infty$. \\

In the first case, the Lemma \ref{F1} implies that 
\begin{eqnarray*}
	C+o_n(1)\|u_{n}\| \geq  C_{1}\biggl[\|\nabla u_{n}\|^{l}_{\Phi}+  \|u_{n}\|^{l}_{\Phi}\biggl]\geq C_{2}\|u_{n}\|^{l}.
\end{eqnarray*}
for $n$ large enough, which is absurd, because $l\geq 1$ and $o_n(1) \to 0$.\\

In case b), we have for $n$ large enough
\begin{eqnarray*}
	C_3+o_n(1)\|\nabla u_{n}\|_{\Phi} \geq  C_{2}\|\nabla u_{n}\|^{l}_{\Phi},
\end{eqnarray*}
which is absurd. The last case is similar to the case b). 
\end{proof}

\begin{lemma} If $(u_n)$ is a $(PS)$ sequence for $I$ with $	u_n \stackrel{*}{\rightharpoonup} u \quad \mbox{in} \quad W^{1,\Phi}(\mathbb{R}^N) $, then there is a subsequence of $(u_n)$, still denoted by itself, such that
$$
u_n(x) \to u(x) \quad \mbox{and} \quad \nabla u_n(x) \to \nabla u(x) \quad \mbox{a.e. in} \quad \mathbb{R}^N. 
$$	
\end{lemma}
\begin{proof} By Sobolev embedding, the embedding $W^{1,\Phi}(\mathbb{R}^N) \hookrightarrow L^{1}(B_R(0))$ is compact for every $R>0$. Thus, for some subsequence of $(u_n)$, still denoted by itself,
$$
u_n(x) \to u(x) \quad \mbox{a.e. in} \quad \mathbb{R}^N. 
$$		
Given $R>0$, let us consider $\xi=\xi_R\in
C_0^\infty(\mathbb{R}^N)$ satisfying 
$$
0\leq\xi\leq 1, \xi \equiv 1 \,\,\, \mbox{in} \,\,\, B_R(0) \,\,\, \mbox{and} \,\,\, \text{supp}(\xi)\subset B_{2R}(0).
$$
Therefore, 
\begin{eqnarray}\label{lim0.1}
	0&\leq& \int_{B_R(0)}\left(\phi(\mid\nabla u_n\mid)\nabla u_n-\phi(\mid
	\nabla u\mid)\nabla u\right)(\nabla u_n-\nabla u)\,dx+\nonumber\\
	&&  \int_{B_R(0)}V(x)\left(\phi(\mid u_n\mid)u_n-\phi(\mid
	 u\mid) u\right)( u_n- u)\,dx\nonumber \\
	& \leq & \int_{B_{2R}(0)}\left( \phi(\mid \nabla u_n\mid)\nabla
	u_n-\phi(\mid\nabla u\mid)\nabla u\right)(\nabla u_n-\nabla u)\xi\,dx+\nonumber\\
	&&\int_{B_{2R}(0)}V(x)\left( \phi(\mid  u_n\mid)
	u_n-\phi(\mid u\mid)u\right)(u_n-u)\xi\,dx \nonumber\\
	&=&\int_{B_{2R}(0)}\phi(\mid \nabla  u_n\mid) \nabla u_n( \nabla u_n-
\nabla 	u)\xi\,dx-\int_{B_{2R}(0)}\phi(\mid  \nabla u\mid) \nabla u(\nabla  u_n-
\nabla	u)\xi\,dx + \nonumber \\
	&&\int_{B_{2R}(0)}V(x)\left( \phi(\mid  u_n\mid)
	u_n-\phi(\mid u\mid)u\right)(u_n-u)\xi\,dx. \nonumber
\end{eqnarray}	

As $\Phi$ satisfies the $\Delta_2$-condition and $(u_n)$ is bounded, it follows that $(\xi(u_n-u))$ is a bounded sequence in $W^{1,\Phi}(\mathbb{R}^N)$. Thereby, $\frac{\partial I(u_n)}{\partial (\xi(u_n-u)) }=o_n(1)$, that is, 
$$
\int_{\mathbb{R}^N}\phi(|\nabla u_n|)\nabla u_n \nabla (\xi(u_n-u))\,dx+	\int_{\mathbb{R}^N}V(x)\phi(| u_n|)u_n (\xi(u_n-u))\,dx-\int_{\mathbb{R}^N}f(u_n)(\xi(u_n-u))\,dx=o_n(1).
$$
Now, it is enough to repeat the same argument explored in \cite[Lemma 4.3]{AGJ} to conclude that for some subsequence
$$
\nabla u_n(x) \to \nabla u(x) \quad \mbox{a.e. in} \quad \mathbb{R}^N. 
$$
	
\end{proof}	

\section{ Proof of Theorem \ref{T1}}

By the previous section there exists a  $(PS)_d$ sequence $(u_n) \subset W^{1,\Phi}(\mathbb{R}^N)$ for $I$, where $d$ is the mountain pass level. Since $(u_n)$ is bounded, we can assume that for some subsequence, there is $u \in L_{loc}^{\Phi}(\mathbb{R}^N)$ such that $u_n \to u$ in $L_{loc}^{\Phi}(\mathbb{R}^N)$. By Corollary \ref{Corolario0}, we derive that $u \in W^{1,\Phi}(\mathbb{R}^N)$. 

Now, recalling that $\left(\displaystyle \int_{\R^N}\phi(|\nabla u_n|)|\nabla u_n|^{2}\,dx\right), \left(\displaystyle \int_{\R^N}\phi(|u_n|)|u_n|^{2}\,dx\right),\left(\displaystyle \int_{\R^N}\Phi(|\nabla u_n|)\,dx\right) $ and $\left(\displaystyle \int_{\R^N}\Phi(|u_n|)\,dx\right)$ are bounded, the identity 
(\ref{DES2D}) ensures that $(\phi(|\nabla u_n|)|\nabla u_n|)$ and $(\phi(|u_n|)|u_n|)$ are bounded sequences in $L^{\tilde{\Phi}}(\mathbb{R}^N)$. Gathering these information, we can apply the Lemma \ref{pontual} with $\Phi$ replaced by $\widetilde{\Phi}$ to obtain  
$$
\int_{\mathbb{R}^N}(\phi(|\nabla u_n|)\nabla u_n \nabla v+V(x)\phi(|u_n|) u_n  v)\,dx \to \int_{\mathbb{R}^N}(\phi(|\nabla u|)\nabla u \nabla v+V(x)\phi(|u|) u  v)\,dx, \quad \forall v \in C_0^{\infty}(\R^N).
$$ 

Now, using the fact that the embeedings of $W^{1,\Phi}(\mathbb{R}^N)$ into $L^{{B}}(B_R(0))$ and $L^{{\Phi}}(B_R(0))$ are compact for every $R>0$, we have  
$$
\int_{\mathbb{R}^N}f(u_n)v\,dx \to \int_{\mathbb{R}^N}f(u)v\,dx, \quad \forall v \in C_0^{\infty}(\R^N).
$$
The last two limits yield 
$$
\int_{\mathbb{R}^N}\phi(|\nabla u|)\nabla u \nabla v\,dx+\int_{\mathbb{R}^N}V(x)\phi(|u|)u v \, dx=\int_{\mathbb{R}^N}f(u)v\,dx, \quad \forall v \in C_0^{\infty}(\R^N).
$$
Now,  the fact that $\phi(|\nabla u|)|\nabla u|, \phi(|u|)|u| \in L^{\widetilde{\Phi}}(\mathbb{R}^N)$ and $f(u) \in L^{\Phi}(\mathbb{R}^N)+L^{B}(\mathbb{R}^N)$ together with the fact that $\overline{C_0^{\infty}(\mathbb{R}^N)}^{\|\,\,\|}=W^{1,\Phi}(\mathbb{R}^N)$ give  
$$
\int_{\mathbb{R}^N}\phi(|\nabla u|)\nabla u \nabla v\,dx+\int_{\mathbb{R}^N}V(x)\phi(|u|)u v \, dx= \int_{\mathbb{R}^N}f(u)v\,dx, \quad \forall v \in W^{1,\Phi}(\mathbb{R}^N),
$$
that is, $u$ is a critical point of $I$ in $W^{1,\Phi}(\mathbb{R}^N)$, and so, $u$ is a weak solution of $(P)$. 

In this point we have the following question: Is $u$ nontrivial? If the answer is yes, we have finished the proof of Theorem \ref{T1}. Otherwise, we must work more a little, and in this case, Theorems  \ref{Lions1} and \ref{lieb} are crucial in our approach. Indeed, we claim that the sequence $(u_n)$ does not satisfy the condition $(*)$ in Theorem \ref{Lions1}, otherwise we must have the convergence
$$
\int_{\mathbb{R}^N}B(|u_n|)\,dx \to 0 \quad \mbox{as} \quad n \to +\infty
$$
that together with $(f_1)$ and $(\phi_3)$ yields 
$$
\int_{\mathbb{R}^N}f(u_n)u_n\,dx \to 0 \quad \mbox{as} \quad n \to +\infty.
$$
Since 
$$
\int_{\mathbb{R}^N}\phi(|\nabla u_n|)|\nabla u_n|^{2}\,dx+\int_{\mathbb{R}^N}V(x)\phi(|u_n|)|u_n|^2 dx=\int_{\mathbb{R}^N}f(u_n)u_n\,dx+o_n(1), 
$$ 
it follows that 
$$
\int_{\mathbb{R}^N}\phi(|\nabla u_n|)|\nabla u_n|^{2}\,dx+\int_{\mathbb{R}^N}V(x)\phi(|u_n|)|u_n|^2 dx \to 0. 
$$
The last limit combines with $(\phi_3)$ to give  
$$
\int_{\mathbb{R}^N}(\Phi(|\nabla u_n|)+\Phi(|u_n|))\,dx \to 0, 
$$
and so, by $\Delta_2$-condition,  
$$
u_n \to 0 \quad \mbox{in} \quad W^{1,\Phi}(\mathbb{R}^N),
$$
which is absurd, because $I(u_n) \to d>0$, see (\ref{sequencia2}).   

From this, there are $\epsilon, \delta>0$ such that 
$$
mes([|u_n|>\epsilon])\geq \delta, \quad \forall n \in \mathbb{N}. 
$$
By Theorem \ref{lieb}, there is $(z_n) \subset \mathbb{Z}^N$ such that $w_n(x)=u_n(x+z_n)$ has a nontrivial limit $w \in L_{loc}^{\Phi}(\mathbb{R}^N)$. Therefore, by Corollary  \ref{Corolario0},  $w \in W^{1,\Phi}(\mathbb{R}^N)$. Moreover, fixed $v \in W^{1,\Phi}(\mathbb{R}^N)$, we have 
$$ 
\int_{\mathbb{R}^N}\phi(|\nabla w_n|)\nabla w_n \nabla v\,dx+\int_{\mathbb{R}^N}\phi(|w_n|) w_n v\,dx= \int_{\mathbb{R}^N}f(w_n)v\,dx+o_n(1).
$$
Arguing as above, we conclude that 
$$ 
\int_{\mathbb{R}^N}\phi(|\nabla w|)\nabla w \nabla v\,dx+\int_{\mathbb{R}^N}\phi(|w|) w v\,dx= \int_{\mathbb{R}^N}f(w)v\,dx, \quad \forall v \in W^{1,\Phi}(\mathbb{R}^N),
$$
showing that $w$ is a nontrivial weak solution to $(P)$.

\end{document}